\documentclass[final]{siamltex}

\title{Condition Numbers of Gaussian Random Matrices
            \thanks{ This research was supported in part by the Applied 
             Mathematical Sciences 
             Research Program of the Office of Mathematical, Information, and
             Computational Sciences, U.S.  Department of Energy under contract
             DE-AC05-00OR22725 with UT-Battelle, LLC}}

\author{Zizhong Chen
          \thanks{Corresponding author's contact information:
                    Department of Computer Science, 
                    The University of Tennessee,
                    203 Claxton Complex,   
                    1122 Volunteer Boulevard, 
                    Knoxville, TN, 37996-3450. E-mail: {\tt  zchen@cs.utk.edu}.
          }
        \and Jack J. Dongarra
        }

\begin{document}

\maketitle

\begin{abstract}
Let $G_{m \times n}$ be an $m \times n$ real random matrix 
whose elements are independent and identically distributed 
standard normal random variables, and let $\kappa_2(G_{m \times n})$ 
be the 2-norm condition number of $G_{m \times n}$.
We prove that, for any $m \geq 2$, $n \geq 2$ and $x \geq |n-m|+1$,
$\kappa_2(G_{m \times n})$ satisfies 
$
 \frac{1}{\sqrt{2\pi}} \left( { c }/{x} \right)^{|n-m|+1} 
< P\left(\frac{\kappa_2(G_{m \times n})} {{n}/{(|n-m|+1)}}> x \right) < 
 \frac{1}{\sqrt{2\pi}} \left( { C }/{x} \right)^{|n-m|+1},
$
where $0.245 \leq c \leq 2.000$ and $ 5.013 \leq C \leq 6.414$ are universal positive
constants independent of $m$, $n$ and $x$. Moreover, 
for any $m \geq 2$ and $n \geq 2$,
$
E(\log\kappa_2(G_{m \times n})) < \log \frac{n}{|n-m|+1} + 2.258.
$
A similar pair of results for complex Gaussian random matrices is also established.
\end{abstract}



\pagestyle{myheadings}
\thispagestyle{plain}
\markboth{ZIZHONG CHEN and JACK J. DONGARRA}{CONDITION NUMBERS OF GAUSSIAN RANDOM MATRICES}

\section{Introduction}


In \cite{Edelman89}, Edelman obtained the limiting distributions 
and the limiting expected logarithms of the condition numbers of  
random rectangular matrices whose elements are independent and 
identically distributed standard normal random variables.
The exact distributions of the condition numbers of $2 \times n$
matrices are also given  in \cite{Edelman89} by Edelman.


However, in the study of real-number and complex-number
error correction codes based on random matrices \cite{Chen104} and
their applications in fault tolerant high performance computing~\cite{Chen204}, 
in order to estimate the numerical stability and reliability of our
coding schemes, we need to estimate the probabilities that the 
condition numbers of small random rectangular matrices 
are large. For example, what is the probability that 
the condition number of a $10 \times 5$ random matrix
is larger than $10^2$? 


In this paper, we investigate the tails of
the condition number distributions of random rectangular matrices
whose elements are independent and identically distributed standard normal 
real or complex random variables. We establish upper and lower bounds for
the tails of the condition number distributions of these matrices.
Upper bounds for the expected
logarithms of the condition numbers of these matrices are also given.

Based on our results, for random rectangular matrices whose elements 
are independent and identically distributed standard 
normal real or complex random variables, we are able to estimate 
the probabilities that their condition numbers are large. For example, based on
our results, we are able to tell, for a $10 \times 5$ real random matrix
whose elements are independent and identically distributed standard 
normal random variables, the probability that the 
condition number is larger than $10^2$ is less than $6 \times 10^{-7}$.

Our main results for the 2-norm condition number $\kappa$ of an $m \times n$ real 
random matrix whose elements are independent and identically distributed standard 
normal random variables are:
$$
 \frac{1}{\sqrt{2\pi}} \left( \frac{c}{x} \right)^{|n-m|+1} 
< P\left(\frac{ \kappa} {{n}/{(|n-m|+1)}}  > x\right) < 
 \frac{1}{\sqrt{2\pi}} \left( \frac{C}{x} \right)^{|n-m|+1},
$$
and
$$
E(\log\kappa) < \log\frac{n}{|n-m|+1} + 2.258,
$$
where $0.245 \leq c\leq 2.000$ and $5.013 \leq C \leq 6.414$ are universal positive
constants independent of $m$, $n$ and $x$, and $m\geq 2$, $n\geq 2$ and
$x \geq |n-m|+1$.

For an $m \times n$ complex 
random matrix whose elements are independent and identically distributed standard 
normal random variables, our main results for the 2-norm condition number $\kappa$ are:
$$
 \frac{1}{{2\pi}} \left( \frac{c}{x} \right)^{2(|n-m|+1)} 
< P\left(\frac{ \kappa} {{n}/{(|n-m|+1)}}  > x\right) < 
 \frac{1}{{2\pi}} \left( \frac{C}{x} \right)^{2(|n-m|+1)},
$$
and
$$
E(\log\kappa) < \log\frac{n}{|n-m|+1} + 2.240,
$$
where $0.319 \leq c\leq 2.000$ and $5.013 \leq C \leq 6.298$ are universal positive
constants independent of $m$, $n$ and $x$, and $m\geq 2$, $n\geq 2$ and
$x \geq |n-m|+1$.

After finishing the manuscript of this paper, we communicated with Edelman and
learned that similar problem was also being studied independently by
Edelman and Sutton~\cite{Edelman05}. After simple formatting, 
the upper bounds in both papers actually can be unified into the same format
$$
P\left( \kappa  > x \right) \leq C(m,n,\beta) \left( \frac{1}{x} \right)^{\beta (|n-m|+1)},
$$
where $\beta=1$ for real random matrices and  $\beta=2$ for complex random matrices, and 
$C(m,n,\beta)$ is a function of $m,n$, and $\beta$. 
However, the function $C(m,n,\beta)$ in the two papers do take very different forms
and imply very different meanings.

On one hand, the bounds in~\cite{Edelman05} are asymptotically tight as $x \rightarrow \infty$ 
while the bounds in this paper are not. On the other hand, the bounds in this paper 
involve only elementary functions. Hence they
are much simpler than the asymptotically tight bounds in~\cite{Edelman05} which 
involve high order moments of the largest eigenvalues of Wishart matrices.
Although for the special case of large square random matrices, simple estimations
for $C(m,n,\beta)$ are given in~\cite{Edelman05},
for general rectangular matrices, no simple estimation is available.

It is well-known that the joint eigenvalue density function of a Wishart matrix 
has a closed form expression~\cite{James64}. Therefore, $P\left( \kappa  > x \right)$
can actually be expressed {\it accurately} as a high dimensional
integration of this joint eigenvalue density function. One of the key aspects
to estimate $P\left( \kappa  > x \right)$ is to find a simple-to-use estimation of
this accurate (but not simple-to-use) high dimensional integral expression. 
This paper is meaningful in that it finds out such a simple-to-use estimation 
by giving out simple upper and lower bounds
which  involve only elementary functions. We refer interested readers to~\cite{Edelman05}
for more accurate asymptotically tight bounds and other related bounds for 
the tails of the condition numbers of general $\beta$-Laguerre ensembles.

Above and in what follows in this paper, the constant $C$ and $c$
denote universal positive constants independent of $m$, $n$ and $x$; 
however, identical symbols may represent different numbers in different place.

\section{Preliminaries and basic facts}

Let $X$ be an $m \times n$ matrix.
If $\sigma_1 \geq \sigma_2\geq ...\geq \sigma_p $,
where $p= \mbox{min} \{m,n\}$, 
are the $p$ singular values of $X$,
then the {\it 2-norm condition number } of $X$ is 
$$\kappa_2(X) = \frac{\sigma_1}{\sigma_p}.$$

For any $m \times n$ matrix $X$, $X^T$ is
an $n \times m$ matrix and $\kappa_{2}(X) = \kappa_{2}(X^T)$. 
So, without loss of generality, in discussing the condition 
numbers of random matrices, it is enough to only consider 
random matrices with no  more rows than columns. Therefore, 
from now on, when we speak of an $m \times n$ matrix, 
we will assume $m \leq n$ in the rest of this paper. 

Let $G_{m \times n}$ be an $m \times n$ real random 
matrix whose elements are independent and 
identically distributed standard normal random variables. 
Let $W_{m,n}$ denote the $m \times m$ random matrix $G_{m \times n} G_{m \times n}^{T}$.
$W_{m,n}$ is the well known Wishart matrix named after John Wishart who has first
studied its distribution.

Similar to \cite{Edelman89}, in this paper, we will study the condition number of $G_{m \times n}$ 
through investigating the eigenvalues of the Wishart matrix $W_{m,n}$. 
The following lemma establishes a simple relationship between
the condition number of $G_{m \times n}$ and the eigenvalues of $W_{m,n}$.
\begin{proposition} 
If  $\lambda_{max}$ is the largest  eigenvalue of $W_{m,n}$,
and $\lambda_{min}$ is the smallest eigenvalue of $W_{m,n}$, then
the 2-norm condition number of $G_{m \times n}$ satisfies
$$
\kappa_2(G_{m \times n})=\sqrt{\frac{\lambda_{max}}{\lambda_{min}}}.
$$
\end{proposition}

Remarkably enough, the exact joint probability density function for 
the $m$ eigenvalues of  the Wishart matrix $W_{m,n}$ 
can be written down in a closed form \cite{James64}.

\begin{lemma} \label{real: joint eigenvalue distribution }
If $\lambda_1 \geq ... \geq \lambda_m $ are the $m$ eigenvalues of $W_{m,n}$, 
then the joint probability density function of $\lambda_1 \geq ... \geq \lambda_m $ 
is
\begin{equation}
f(x_1, ..., x_m) =  K_{m,n} e^{-\frac{1}{2} \sum_{i=1}^{m} x_i } 
\prod_{i=1}^{m} x_i^{\frac{1}{2}(n-m-1)} 
\prod_{i=1}^{m-1} \prod_{j=i+1}^{m} ( x_i - x_j ), 
\end{equation}
where
\begin{equation}
K_{m,n}^{-1} = \left( \frac{2^n}{\pi} \right)^{m/2 } \prod_{i=1}^{m} \Gamma\left(\frac{n-m+i}{2}\right) 
\Gamma\left(\frac{i}{2}\right).
\end{equation}
\end{lemma}

Let $N(0,1)$ denote the standard normal distribution.
Let $\widetilde{N}(0,1)$ denote the
distribution of $u + i v$, where $u$ and $v$ are independent and identically distributed
$N(0,1)$ random variables, and $i=\sqrt{-1}$. 
Let $\widetilde{G}_{m \times n}$ be an $m \times n$ complex random matrix whose elements are independent and 
identically distributed $\widetilde{N}(0,1)$ random variables. 
Let $\widetilde{W}_{m,n}$ denote the $m \times m$ random matrix 
$\widetilde{G}_{m \times n} \widetilde{G}_{m \times n}^{H}$.
In literature, $\widetilde{W}_{m,n}$ is called the complex Wishart matrix.

Similar to the real case, there is also a simple relationship between
the condition number of $\widetilde{G}_{m \times n}$ and the eigenvalues of $\widetilde{W}_{m,n}$.
\begin{proposition} 
If  $\widetilde{\lambda}_{max}$ is the largest  eigenvalue of $\widetilde{W}_{m,n}$,
and $\widetilde{\lambda}_{min}$ is the smallest eigenvalue of $\widetilde{W}_{m,n}$, then
the 2-norm condition number of $\widetilde{G}_{m \times n}$ satisfies
$$
\kappa_2(\widetilde{G}_{m \times n})=\sqrt{\frac{\widetilde{\lambda}_{max}}{\widetilde{\lambda}_{min}}}.
$$
\end{proposition}

Like the real case, the exact joint probability density function for 
the $m$ eigenvalues of  the complex Wishart matrix $\widetilde{W}_{m,n}$ 
can also be written down in a closed form \cite{James64}.
\begin{lemma} \label{complex: joint eigenvalue distribution }
If $\widetilde{\lambda}_1 \geq ... \geq \widetilde{\lambda}_m $ are the $m$ eigenvalues of $\widetilde{W}_{m,n}$, 
then the joint probability density function of 
$\widetilde{\lambda}_1 \geq ... \geq \widetilde{\lambda}_m $ is
\begin{equation}
\widetilde{f}(x_1, ..., x_m) =  \widetilde{K}_{m,n} e^{-\frac{1}{2} \sum_{i=1}^{m} x_i } 
\prod_{i=1}^{m} x_i^{n-m} 
\prod_{i=1}^{m-1} \prod_{j=i+1}^{m} ( x_i - x_j )^2, 
\end{equation}
where
\begin{equation}
\widetilde{K}_{m,n}^{-1} = 2^{m n} \prod_{i=1}^{m} \Gamma\left(n-m+i\right) 
\Gamma\left(i\right).
\end{equation}
\end{lemma}

In the process of deriving our upper and lower bounds for the tails
of the condition number distributions, some bounds for Gamma and incomplete
Gamma functions are very useful.

\begin{lemma} \label{integral inequality A}
Assume $ a > 0$, and $b > 0$. If $t \leq \frac{b}{a} $, then
$$
\int_0^t e^{-a x} x^b d x  \leq e^{-a t} t^{b+1}.
$$
\end{lemma}
\begin{proof} 
Let
$
f(t) = \int_0^t e^{- a x} x^b d x - e^{- a t} t^{b+1},
$
then
$
f^\prime(t) = e^{-a t} t^b ( 1 + at - (b+1) ).
$
So $f(t)$ decreases on $[0,\frac{b}{a}]$ and increases
on $[\frac{b}{a}, \infty ).$ Since
$
f(0)=0,
$
and
$
f(\infty) = \int_0^\infty e^{- a x} x^b d x > 0,
$
if $t \leq \frac{b}{a} $, then $f(t) < 0$. 
Therefore, if $t \leq \frac{b}{a} $, then
$
\int_0^t e^{-a x} x^b d x  \leq e^{-a t} t^{b+1}.
$
\qquad
\end{proof}

\begin{lemma} \label{integral inequality B}
Assume $ a > 0$, $b > 0$, and $  k > \frac{1}{a}$. If $t \geq \frac{kb}{ka-1}$, then
$$
\int_t^\infty e^{-ax} x^b d x \leq k e^{-at} t^{b}.
$$
\end{lemma}
\begin{proof} 
Let
$
f(t) = \int_t^\infty e^{-ax} x^b d x - k e^{-at} t^{b},
$
then
$ 
f^\prime(t) = e^{-at} t^b (-1 + ka -\frac{kb}{t}).
$
So $f(t)$ decreases on $[0, \frac{kb}{ka-1}]$ and increases
on $[\frac{kb}{ka-1}, \infty ).$ Since
$
f(0)=\int_0^\infty e^{- a x} x^b d x > 0,
$
and
$
f(\infty) = 0.
$
So, if $t \geq \frac{kb}{ka-1}$, then $f(t) < 0$. 
Therefore, if $t \leq \frac{kb}{ka-1}$, then
$
\int_t^\infty e^{-ax} x^b d x \leq k e^{-at} t^{b}.
$
\qquad
\end{proof}

\begin{lemma} \label{strling's formula}
If $\Gamma(x)=\int_0^\infty e^{-t} t^{x-1} d t$, where $x>0$, then 
\begin{equation}
\sqrt{2 \pi} x^{x+\frac{1}{2}} e^{-x}
< \Gamma(x+1) < 
\sqrt{2 \pi} x^{x+\frac{1}{2}} e^{-x+\frac{1}{12 x}},
\end{equation}
and
\begin{equation}
\Gamma(x+\frac{1}{2}) <  \Gamma(x) \sqrt{x}.
\end{equation}
\end{lemma}
\begin{proof} 
(2.5) follows straightforwardly from 6.1.38 in~\cite{Abram70}, and (2.6) can be obtained from
answer to Problem 9.60 in~\cite{Graham94}.
\qquad
\end{proof}





\section{Bounds for eigenvalue densities of Wishart matrices}

In this section, we will prove some bounds for the probability density
functions of the eigenvalues of Wishart matrices.
These bounds are very useful in the derivation of the bounds
for the tails of the condition number distributions.

Let $\lambda_{max}$ denote the largest  eigenvalue of $W_{m,n}$,
and $\lambda_{min}$ denote the smallest eigenvalue of $W_{m,n}$.
In the following lemma, we prove an upper bound for the joint probability density
function of $\lambda_{max}$ and $\lambda_{min}$.

\begin{lemma} \label{real: upper bound for joint eigenvalue}
Let  $ f_{\lambda_{max}, \lambda_{min}} (x,y) $ denote the joint
probability density function of $\lambda_{max}$ and $\lambda_{min}$, then 
$ f_{\lambda_{max}, \lambda_{min}} (x,y) $ satisfies:
\begin{eqnarray}
f_{\lambda_{max}, \lambda_{min}} (x,y)  \leq   
C_{m,n} e^{-\frac{1}{2} (x + y) } x^{\frac{1}{2}(n+m-3)}
y^{\frac{1}{2}(n-m-1)},
\end{eqnarray}
where 
\begin{eqnarray}
C_{m,n} = \frac{1}{4\Gamma \left(m-1\right) \Gamma\left(n-m+1\right) }.
\end{eqnarray}
\end{lemma}

\begin{proof}
Let  $R_{x,y} = \{(x_2, x_3, . . ., x_{m-1}): 
                   x \geq x_2 \geq . . . \geq x_{m-1} \geq y \} \subseteq R^{m-2}$.
From the joint probability density function of the $m$ eigenvalues of $W_{m,n}$ in Lemma 2.2, we have
\begin{eqnarray}
f_{\lambda_{max}, \lambda_{min}} (x,y)
& = &  \int_{R_{x,y}} f(x, x_2, ..., x_{m-1}, y) d x_2 d x_3 ... d x_{m-1} \nonumber \\ 
& = &  K_{m,n} e^{-\frac{1}{2}(x+y) } x^{\frac{1}{2}(n-m-1)} y^{\frac{1}{2}(n-m-1)} \nonumber \\ & & \quad 
       \int_{R_{x,y}} e^{-\frac{1}{2} \sum_{i=2}^{m-1}x_i }  
       \prod_{i=2}^{m-1} x_i^{\frac{1}{2}(n-m-1)}   
              \nonumber \\[-1.5ex]
             \label{eq3.10}\\[-1.5ex] & & \qquad 
       (x-y) \prod_{i=2}^{m-1} ( x - x_i )  ( x_i - y )  
       \prod_{i=2}^{m-2} \prod_{j=i+1}^{m-1} ( x_i - x_j ) \prod_{i=2}^{m-1}d x_i.\nonumber
\end{eqnarray}

Let $R_{m-2} = \{(x_2, x_3, . . ., x_{m-1}): x_2 \geq . . . \geq x_{m-1} \geq 0\}$,
then $R_{m-2} \subseteq R_{x,y} $. Note that, in (3.3), $x \geq x_i \geq y$ for $i=2, 3, ..., m-1$.
Replacing $x-y$ and $x - x_i$ by $x$, and $x_i - y $ by $x_i$ for 
$i=2, 3, ..., m-1$, and $R_{x,y}$ by $R_{m-2}$, then we get
\begin{eqnarray}
f_{\lambda_{max}, \lambda_{min}} (x,y)
& \leq &   K_{m,n} e^{-\frac{1}{2}(x+y) } x^{\frac{1}{2}(n+m-3)} y^{\frac{1}{2}(n-m-1)}
           \nonumber \\
          & & \quad 
          \int_{R_{m-2}} e^{-\frac{1}{2} \sum_{i=2}^{m-1}x_i }  
          \prod_{i=2}^{m-1} x_i^{\frac{1}{2}(n-m+1)}          
          \prod_{i=2}^{m-2} \prod_{j=i+1}^{m-1} ( x_i - x_j ) \prod_{i=2}^{m-1}d x_i. \label{eq2.13} 
\end{eqnarray}

Note that $f(x_1, x_2, ..., x_m)$ in (2.1) is 
a probability density function, therefore, for any $m \leq n$, we have
\begin{eqnarray*}
\int_{R_{m}}  e^{-\frac{1}{2} \sum_{i=1}^{m} x_i } 
\prod_{i=1}^{m} x_i^{\frac{1}{2}(n-m-1)} 
\prod_{i=1}^{m-1} \prod_{j=i+1}^{m} ( x_i - x_j ) \prod_{i=1}^{m} dx_i
& = &   K_{m,n}^{-1}, 
\end{eqnarray*}
where $R_{m} = \{x_1 \geq x_2 \geq ... \geq x_{m} \geq 0 \} \subseteq R^{m}$. 
Therefore, we have
\begin{eqnarray}
\int_{R_{m-2}} e^{-\frac{1}{2} \sum_{i=2}^{m-1}x_i }  
\prod_{i=2}^{m-1} x_i^{\frac{1}{2}(n-m+1)}          
\prod_{i=2}^{m-2} \prod_{j=i+1}^{m-1} ( x_i - x_j ) \prod_{i=2}^{m-1}d x_i
& = &   K_{m-2,n}^{-1}. 
\end{eqnarray}
Substitute (3.5) into (3.4), we obtain
\begin{eqnarray}
f_{\lambda_{max}, \lambda_{min}} (x,y)
& \leq & \frac{ K_{m,n} }{ K_{m-2,n} }  
         e^{-\frac{1}{2}(x+y) } x^{\frac{1}{2}(n+m-3)} y^{\frac{1}{2}(n-m-1)}.
\end{eqnarray}

From (2.2), we have 
\begin{eqnarray}
\frac{ K_{m,n} }{ K_{m-2,n} } 
& = & \frac{\pi}{2^n} \frac{1}{ \Gamma \left( \frac{m-1}{2}\right) \Gamma \left( \frac{m}{2}\right)   
      \Gamma \left( \frac{n-m+1}{2}\right) \Gamma \left( \frac{n-m+2}{2}\right) }  
             \nonumber \\[-1.5ex]
             \label{eq2.14}\\[-1.5ex]
& = & \frac{1}{4\Gamma \left(m-1\right) \Gamma\left(n-m+1\right) }. \nonumber
\end{eqnarray}
Substitute (3.6) into (3.5), we get (3.1) and (3.2) .
\qquad
\end{proof}

Let $\widetilde{\lambda}_{max}$ denote the largest  eigenvalue of $\widetilde{W}_{m,n}$,
and $\widetilde{\lambda}_{min}$ denote the smallest eigenvalue of $\widetilde{W}_{m,n}$.
Similar to the real case, in the following lemma, 
we give an upper bound for the joint probability density
function of $\widetilde{\lambda}_{max}$ and $\widetilde{\lambda}_{min}$.
The upper bound in complex case can be proved using the same techniques used in the real case. 
Therefore, we omit the proof and only give the result here.

\begin{lemma} \label{complex: upper bound for joint eigenvalue}
Let  $ \widetilde{f}_{ {\widetilde{\lambda}}_{max}, {\widetilde{\lambda}}_{min} } (x,y) $ denote the joint
probability density function of $\widetilde{\lambda}_{max}$ and $\widetilde{\lambda}_{min}$, then 
$ \widetilde{f}_{ {\widetilde{\lambda}}_{max}, {\widetilde{\lambda}}_{min} } (x,y) $ satisfies:
\begin{eqnarray}
\widetilde{f}_{ \widetilde{\lambda}_{max},\widetilde{\lambda}_{min} } (x,y)   \leq   
\widetilde{C}_{m,n} e^{-\frac{1}{2} (x + y) } x^{n+m-2}
y^{n-m},
\end{eqnarray}
where 
\begin{eqnarray}
\widetilde{C}_{m,n} = \frac{1}{2^{2n} \Gamma(m-1)\Gamma(m)\Gamma(n-m+1)\Gamma(n-m+2) }.
\end{eqnarray}
\end{lemma}

Bounds for the probability density functions of the smallest eigenvalues
are also very useful in the derivation of the bounds
for the tails of the condition number distributions.
In the following lemma, we prove upper and
lower bounds for the probability density function of the smallest
eigenvalue of a real Wishart matrix.

\begin{lemma} \label{real: upper bound for smallest eigenvalue}
Let  $ f_{\lambda_{min}} (x) $ denotes the probability 
density function of the smallest eigenvalue of $W_{m,n}$, then 
$ f_{\lambda_{min}} (x) $ satisfies:
\begin{eqnarray}
L_{m,n} e^{-\frac{m}{2}x } x^{\frac{1}{2}(n-m-1)}
\leq  f_{\lambda_{min}} (x) \leq 
L_{m,n} e^{-\frac{1}{2}x } x^{\frac{1}{2}(n-m-1)},
\end{eqnarray}
where
\begin{eqnarray}
L_{m,n} = \frac{2^{\frac{n-m-1}{2}}  \Gamma \left( \frac{n+1}{2}\right)   }
            { \Gamma \left( \frac{m}{2}\right)\Gamma\left(n-m+1\right) }. 
\end{eqnarray}
\end{lemma}

\begin{proof}
Let  $R_{x} = \{(x_1, x_2, . . ., x_{m-1}): 
                 x_1 \geq . . . \geq x_{m-1} \geq x \} \subseteq R^{m-1}$.
From the joint probability density function of the eigenvalues of $W_{m,n}$ in Lemma 2.2, we have
\begin{eqnarray*}
f_{\lambda_{min}} (x)
& = &  \int_{R_x} f(x_1, x_2, ..., x_{m-1}, x) dx_1 dx_2 dx_{m-1} \nonumber \\
& = &  K_{m,n} e^{-\frac{1}{2}x } x^{\frac{1}{2}(n-m-1)}
       \int_{R_x} e^{-\frac{1}{2} \sum_{i=1}^{m-1}x_i } 
       \prod_{i=1}^{m-1} x_i^{\frac{1}{2}(n-m-1)} \label{eq3.12}
        \\ & & \quad \prod_{i=1}^{m-1} ( x_i - x )
       \prod_{i=1}^{m-2} \prod_{j=i+1}^{m-1} ( x_i - x_j )  \prod_{i=1}^{m-1} dx_i. \nonumber 
\end{eqnarray*}

For the lower bound part, taking the transformation 
$y_i=x_i-x$, where $i=1, 2, ..., m-1$, we have
\begin{eqnarray*}
f_{\lambda_{min}} (x)
& = &  K_{m,n} e^{-\frac{m}{2}x } x^{\frac{1}{2}(n-m-1)}
       \int_{R_y}  e^{-\frac{1}{2} \sum_{i=1}^{m-1} y_i } 
       \prod_{i=1}^{m-1} (y_i+x)^{\frac{1}{2}(n-m-1)} \\ & & \quad 
       \prod_{i=1}^{m-1} y_i
       \prod_{i=1}^{m-2} \prod_{j=i+1}^{m-1} ( y_i - y_j )  \prod_{i=1}^{m-1} d y_i,
\end{eqnarray*}
where $R_y = \{y_1 \geq y_2 \geq ... \geq y_{m-1} \geq 0 \} \subseteq R^{m-1}$. 

Replacing $ y_i+x$ by $ y_i $ for $i=1, 2, ..., m-1$, we obtain
\begin{eqnarray*}
f_{\lambda_{min}} (x)
&\geq& K_{m,n} e^{-\frac{m}{2}x } x^{\frac{1}{2}(n-m-1)}
       \int_{R_y}  e^{-\frac{1}{2} \sum_{i=1}^{m-1} y_i } 
       \prod_{i=1}^{m-1} y_i^{\frac{1}{2}(n-m+1)} \\ & & \quad 
       \prod_{i=1}^{m-2} \prod_{j=i+1}^{m-1} ( y_i - y_j )  \prod_{i=1}^{m-1} d y_i.
\end{eqnarray*}
Note that 
\begin{eqnarray*}
\int_{R_y}  e^{-\frac{1}{2} \sum_{i=1}^{m-1} y_i } 
\prod_{i=1}^{m-1} y_i^{\frac{1}{2}(n-m+1)}
\prod_{i=1}^{m-2} \prod_{j=i+1}^{m-1} ( y_i - y_j )  \prod_{i=1}^{m-1} d y_i
& = &  K_{m-1,n+1} ^{-1}. 
\end{eqnarray*}
Therefore, we obtain
\begin{eqnarray}
f_{\lambda_{min}} (x)
&\geq& \frac{K_{m,n}} {K_{m-1,n+1}} e^{-\frac{m}{2}x } x^{\frac{1}{2}(n-m-1)}.
\end{eqnarray}

For the upper bound part, from \cite{Edelman88}, we have
\begin{eqnarray}
f_{\lambda_{min}} (x)
& \leq & \frac{ K_{m,n}}{ K_{m-1,n+1}}  e^{-\frac{1}{2}x } x^{\frac{1}{2}(n-m-1)}.
\end{eqnarray}

From (2.2), we have 
\begin{eqnarray}
\frac{ K_{m,n} }{ K_{m-1,n+1} } 
& = & \frac{ \sqrt\pi \left(\frac{1}{2}\right)^{\frac{n-m+1}{2}}  \Gamma \left( \frac{n+1}{2}\right)   }
            { \Gamma \left( \frac{m}{2}\right) \Gamma \left( \frac{n-m+1}{2}\right) 
              \Gamma \left( \frac{n-m+2}{2}\right)} 
                  \nonumber \\[-1.5ex]
                  \label{eq3.14}\\[-1.5ex]
& = & \frac{2^{\frac{n-m-1}{2}}  \Gamma \left( \frac{n+1}{2}\right)   }
            { \Gamma \left( \frac{m}{2}\right)\Gamma\left(n-m+1\right) }. \nonumber
\end{eqnarray}
Substitute (3.14) into (3.13) and (3.12), we get (3.10) and (3.11) .
\qquad
\end{proof}

Similar to the real case, in the following lemma, 
we give upper and lower bounds for the probability density
function of the smallest eigenvalue $\widetilde{\lambda}_{min}$ of 
a complex Wishart matrix. These bounds can be proved using the 
same techniques used in the real case. 
Therefore, we omit the proof and only give the result here.

\begin{lemma} \label{complex: upper bound for smallest eigenvalue}
Let  $\widetilde{f}_{\widetilde{\lambda}_{min}} (x) $ denotes the probability 
density function of the smallest eigenvalue of $\widetilde{W}_{m,n}$, then 
$ \widetilde{f}_{\widetilde{\lambda}_{min}} (x) $ satisfies:
\begin{eqnarray}
\widetilde{L}_{m,n} e^{-\frac{m}{2}x } x^{n-m}
\leq  \widetilde{f}_{\widetilde{\lambda}_{min}} (x) \leq 
\widetilde{L}_{m,n} e^{-\frac{1}{2}x } x^{n-m},
\end{eqnarray}
where
\begin{eqnarray}
\widetilde{L}_{m,n} = \frac{ \Gamma(n+1) }
            { 2^{n-m+1}  \Gamma(m) \Gamma(n-m+1)\Gamma(n-m+2) }. 
\end{eqnarray}
\end{lemma}

\section{The upper bounds for the distribution tails}

In this section, we will derive the upper bounds for the tails of 
the condition number distributions of random rectangular matrices
whose elements are independent and identically distributed 
standard normal random variables. Our main results are Theorem 4.5
for real random matrices, and Theorem 4.6 for complex random matrices.

\begin{lemma} \label{Real case: prob inequality A}
For any $A>0$, $x> 0$, and $n \geq m \geq 2 $, the largest eigenvalue $\lambda_{max}$ 
and the smallest eigenvalue $\lambda_{min}$ of $W_{m,n}$ satisfy
\begin{eqnarray*}
P \left( \frac{\lambda_{max}}{\lambda_{min}} > x^2, \lambda_{min} \leq \frac{A^2 n}{x^2} \right) 
< \frac{1}{\Gamma(n-m+2)}\left( \frac{An}{x} \right)^{n-m+1}. 
\end{eqnarray*}
\end{lemma}
\begin{proof}
From the upper bound for the probability density function of $\lambda_{min}$ in Lemma 3.2, we have
\begin{eqnarray*}
P \left( \frac{\lambda_{max}}{\lambda_{min}} > x^2, \lambda_{min} \leq \frac{A^2 n}{x^2} \right) 
& < & P \left(\lambda_{min} \leq \frac{A^2 n}{x^2} \right) \\
& = & \int_0^{\frac{A^2 n}{x^2}} f_{\lambda_{min}}(t) d t  \\
& < &  L_{m,n}\int_0^{\frac{A^2 n}{x^2}} t^{\frac{1}{2}(n-m-1)} d t\\
& = &   \frac{\Gamma \left( \frac{n+1}{2}\right)   }
             { \Gamma \left( \frac{m}{2}\right) \left(\frac{n}{2}\right)^{\frac{n-m+1}{2}}} 
        \frac{1}{\Gamma(n-m+2)}\left( \frac{An}{x} \right)^{n-m+1}.
\end{eqnarray*}

Since $m \leq n$, by applying (2.6) repeatedly, we can prove
\begin{eqnarray*}
\Gamma \left( \frac{m}{2}\right) \left(\frac{n}{2}\right)^{\frac{n-m+1}{2}}
& > & \Gamma \left( \frac{n+1}{2}\right). 
\end{eqnarray*}
Therefore, we have
\begin{eqnarray*}
P \left( \frac{\lambda_{max}}{\lambda_{min}} > x^2, \lambda_{min} \leq \frac{A^2 n}{x^2} \right) 
& < &   \frac{1}{\Gamma(n-m+2)}\left( \frac{An}{x} \right)^{n-m+1}.
\end{eqnarray*}
\qquad
\end{proof}

Similar to real random matrices, for complex random matrices, we have
the following Lemma 4.2. Lemma 4.2 can be proved using the 
same techniques as Lemma 4.1, so we 
will omit the proof and only give the result. 
\begin{lemma} \label{Complex case: prob inequality A}
For any $A > 0$, $ x > 0$, and $n \geq m \geq 2 $,  the largest eigenvalue $\widetilde{\lambda}_{max}$ 
and the smallest eigenvalue $\widetilde{\lambda}_{min}$ of 
$\widetilde{W}_{m,n}$ satisfy
\begin{eqnarray*}
 P \left( \frac{\widetilde{\lambda}_{max}}{\widetilde{\lambda}_{min}} > x^2,
\widetilde{\lambda}_{min} \leq \frac{A^2 n}{x^2} \right) 
< \frac{1}{\Gamma(n-m+2)^2}\left( \frac{A^2n^2}{2x^2} \right)^{n-m+1}. 
\end{eqnarray*}
\end{lemma}

The proof of the following Lemma 4.3 is based on 
the upper bound for the joint probability density 
function of $\lambda_{max}$ and $\lambda_{min}$ 
in Lemma 3.1 and the upper bound of the incomplete
Gamma function in Lemma 2.6.

\begin{lemma} \label{Real case: prob inequality B}
For any $A \geq 2.32$, $x>0$, and $n \geq m \geq 2 $, the largest eigenvalue $\lambda_{max}$ 
and the smallest eigenvalue $\lambda_{min}$ of  $W_{m,n}$  satisfy
\begin{eqnarray*}
P \left(\frac{\lambda_{max}}{\lambda_{min}} > x^2, \lambda_{min} > \frac{A^2 n}{x^2} \right)
 < 0.017 \frac{1}{\Gamma(n-m+2)} \left( \frac{A n}{x} \right)^{n-m+1}.  
\end{eqnarray*}
\end{lemma}
 
\begin{proof}
From the upper bound for the joint probability density function of $\lambda_{max}$
and $\lambda_{min}$ in Lemma 3.1, we have
\begin{eqnarray*}
P \left(\frac{\lambda_{max}}{\lambda_{min}} > x^2, \lambda_{min} > \frac{A^2 n}{x^2} \right) 
&   =   & \int_{\frac{A^2 n}{x^2}}^\infty \int_{t x^2}^\infty  
               f_{\lambda_{max},\lambda_{min}}(s, t) d s d t \\
&   <   & \int_{\frac{A^2 n}{x^2}}^\infty \int_{t x^2}^\infty  
          C_{m,n}  e^{-\frac{1}{2}t } t^{\frac{1}{2}(n-m-1)} 
            e^{-\frac{1}{2}s } s^{\frac{1}{2}(n+m-3)}    d s d t.  
\end{eqnarray*}
Taking the transform $u = t x^2$, we have
\begin{eqnarray*}
P \left(\frac{\lambda_{max}}{\lambda_{min}} > x^2, \lambda_{min} > \frac{A^2 n}{x^2} \right)
&   =   & C_{m,n}  \left( \frac{1}{x} \right)^{n-m+1} 
          \int_{A^2 n}^\infty  e^{-\frac{u}{2 x^2}} u^{\frac{1}{2}(n-m-1)} \\ & & \qquad 
          \left( \int_u^\infty e^{-\frac{1}{2}s } s^{\frac{1}{2}(n+m-3)} ds \right) d u.
\end{eqnarray*}
According to Lemma 2.6, with $k=4$,if $u \geq 2(n+m-3)$, then
\begin{eqnarray*}
\int_u^\infty e^{-\frac{1}{2}s } s^{\frac{1}{2}(n+m-3)} ds
& \leq & 4 e^{-\frac{1}{2}u } u^{\frac{1}{2}(n+m-3)}.
\end{eqnarray*}
Since $A \geq 2.32$ and $n \geq m$, hence, $u \geq A^2n \geq 2(n+m-3)$. Therefore, we have
\begin{eqnarray*} 
P \left(\frac{\lambda_{max}}{\lambda_{min}} > x^2, \lambda_{min} > \frac{A^2 n}{x^2} \right)
& \leq &  4 C_{m,n} \left( \frac{1}{x} \right)^{n-m+1}  
          \int_{A^2 n}^\infty  e^{-\frac{u}{2 x^2}-\frac{1}{2}u} u^{n-2} d u \\
& \leq &  4 C_{m,n} \left( \frac{1}{x} \right)^{n-m+1}  
          \int_{A^2 n}^\infty  e^{-\frac{1}{2}u} u^{n-2} d u. 
\end{eqnarray*}
Since $A \geq 2.32$, so $A^2 n \geq 4(n-2)$. Apply Lemma 2.6 again,  we have
\begin{eqnarray}
P \left(\frac{\lambda_{max}}{\lambda_{min}} > x^2, \lambda_{min} > \frac{A^2 n}{x^2} \right)
& \leq &  16 C_{m,n}
          e^{-\frac{1}{2}A^2 n} A^{2n-4} n^{n-2}  \left( \frac{1}{x} \right)^{n-m+1}  \nonumber \\
&   =   &\frac{4 e^{-\frac{A^2 n}{2} } A^{2n-4} n^{m-3} }{\Gamma(m-1)\Gamma(n-m+1) } 
                       \left( \frac{n}{x} \right)^{n-m+1}  \nonumber \\
&\leq   &\frac{4 e^{(2 \ln A-\frac{A^2}{2})n }}{A^4} \frac{n^{m-2} }{\Gamma(m-1) } 
                      \frac{1}{\Gamma(n-m+2)} \left( \frac{n}{x} \right)^{n-m+1}.    
\end{eqnarray}
Note that, for any $2 \leq m \leq n$, it can be proved that 
\begin{eqnarray}
\frac{n^{m-2}}{\Gamma(m-1) }
&   <   & \frac {e^{n}}{\sqrt{4 \pi}}.  
\end{eqnarray}
Substitute (4.2) into (4.1), we have
\begin{eqnarray*}
P \left(\frac{\lambda_{max}}{\lambda_{min}} > x^2, \lambda_{min} > \frac{A^2 n}{x^2} \right)
& \leq & \frac{4 e^{(2 \ln A-\frac{A^2}{2} +1 )n }}{\sqrt{4 \pi} A^4}
                      \frac{1}{\Gamma(n-m+2)} \left( \frac{n}{x} \right)^{n-m+1}.   
\end{eqnarray*}
Since $A \geq 2.32 $, therefore, we have 
\begin{eqnarray*}
e^{(2 \ln A-\frac{A^2}{2} +1 )n } < 1.
\end{eqnarray*}
Therefore, when $A \geq 2.32 $, we have
\begin{eqnarray*}
P \left(\frac{\lambda_{max}}{\lambda_{min}} > x^2, \lambda_{min} > \frac{A^2 n}{x^2} \right)
& \leq & \frac{4}{\sqrt{4 \pi} A^4} \frac{1}{\Gamma(n-m+2)} \left( \frac{n}{x} \right)^{n-m+1}  \\
& \leq & 0.017 \frac{1}{\Gamma(n-m+2)} \left( \frac{An}{x} \right)^{n-m+1}. 
\end{eqnarray*}
\qquad
\end{proof}

Similar to real random matrices, for complex random matrices, we have
the following Lemma 4.4. Lemma 4.4 can be proved using the 
same techniques as Lemma 4.3, so we 
will omit the proof and only give the result.
 
\begin{lemma} \label{Complex case: prob inequality B}
For any $A \geq 3.2735$, $x>0$, and $n \geq m \geq 2 $, 
the largest eigenvalue $\widetilde{\lambda}_{max}$ 
and the smallest eigenvalue $\widetilde{\lambda}_{min}$ of  $\widetilde{W}_{m,n}$ satisfy
\begin{eqnarray*}
P \left( \frac{\widetilde{\lambda}_{max}}{\widetilde{\lambda}_{min}} > x^2,
\widetilde{\lambda}_{min} > \frac{A^2 n}{x^2} \right) 
<  0.0016 \frac{1}{\Gamma(n-m+2)^2} \left( \frac{A^2 n^2}{2x} \right)^{n-m+1}.  
\end{eqnarray*}
\end{lemma}

We are now prepared to prove our first
main result about the condition numbers 
of real random matrices whose elements are 
independent and identically distributed standard 
normal random variables.

\begin{theorem} \label{ Real Case: new upper bounds for tail }
For any $n \geq m \geq 2$ and $x \geq n-m+1$, the 2-norm 
condition number of $G_{m \times n}$ satisfies
\begin{eqnarray}
P\left(\frac{\kappa_2(G_{m \times n}) } {{n}/{(n-m+1)}}  > x\right) < 
\frac{1}{\sqrt{2\pi}} \left( \frac{C}{x} \right)^{n-m+1},
\end{eqnarray}
where $C \leq 6.414$ is a universal positive constant independent 
of $m$, $n$, and $x$.
\end{theorem}

\begin{proof}
For any $L > 0$, inspired by \cite{Azais03},
we first break down $P( \kappa_2(G_{m \times n}) > x )$ into two parts.
\begin{eqnarray*}
P( \kappa_2(G_{m \times n}) > x ) 
&  =  &  P \left( \frac{\lambda_{max}}{\lambda_{min}} > x^2 \right)   \\
&  =  &  P \left(\frac{\lambda_{max}}{\lambda_{min}} > x^2, \lambda_{min} \leq \frac{L^2 n}{x^2} \right)
       + P \left(\frac{\lambda_{max}}{\lambda_{min}} > x^2, \lambda_{min} > \frac{L^2 n}{x^2} \right).
\end{eqnarray*}
Let $L = 2.32$, then based on Lemma 4.1 and Lemma 4.3, we can get
\begin{eqnarray*}
P( \kappa_2(G_{m \times n}) > x ) 
&  <  &  \frac{1}{ \Gamma(n-m+2) } \left( \frac{ L  n}{x} \right)^{n-m+1} \\ & & \qquad
         + 0.017 \frac{1}{ \Gamma(n-m+2) } \left( \frac{ L  n}{x} \right)^{n-m+1}  \\
&  <  &  \frac{1}{ \Gamma(n-m+2) } \left( \frac{ 1.017 L  n}{x} \right)^{n-m+1}.  
\end{eqnarray*}
Note that, from Lemma 2.7, we have
\begin{eqnarray*}
\Gamma(n-m+2) > \sqrt{2\pi (n-m+1)} (n-m+1)^{n-m+1} e^{-(n-m+1)}.
\end{eqnarray*}
Therefore, we have
\begin{eqnarray*}
P( \kappa_2(G_{m \times n}) > x ) 
&  <  & \frac{1}{\sqrt{2\pi (n-m+1)}}\left( \frac{ 1.017eL \frac{n}{n-m+1} }{x} \right)^{n-m+1}.  
\end{eqnarray*}
Therefore
\begin{eqnarray*}
P\left(\frac{\kappa_2(G_{m \times n}) } {{n}/{(n-m+1)}}  > x\right)
&  <  & \frac{1}{\sqrt{2\pi (n-m+1)}}\left( \frac{ 1.017eL }{x} \right)^{n-m+1}  \\
&  <  & \frac{1}{\sqrt{2\pi}} \left( \frac{ 6.414 }{x} \right)^{n-m+1}. 
\end{eqnarray*}
Let $C = 6.414$, then we get (4.3).
\qquad
\end{proof}

Remark:
\begin{remunerate}

\item The upper bound in Theorem 4.5 is for arbitrary $n \geq m \geq 2$ 
and $x \geq n-m+1$. 
For some special case of $m$ and $n$, more precise upper bound can be obtained. 
For example, for the special case of real random $2 \times n$ matrices, 
based on the exact probability density function of $\kappa_2(G_{2 \times n})$ 
in \cite{Edelman89}, we can get
\begin{eqnarray*}
P\left(\kappa_2(G_{2 \times n}) > x \right) = \left( \frac{2x}{x^2+1} \right)^{n-1}
< \left( \frac{2}{x} \right)^{n-1}.
\end{eqnarray*}

\item For the special case of real random $m \times m$ matrices,
where $m \geq 3$, 
it has been proved in \cite{Azais03} that
\begin{eqnarray}
P \left(\kappa_2(G_{m \times m})   >  m \, . \, x  \right)
< \frac{C^\prime }{x},
\end{eqnarray}
where $C^\prime \leq 5.60$ is a universal positive constant independent 
of $x$ and $m$.
 
In Theorem 4.5, if we take $m = n$, then we have
\begin{eqnarray*}
P\left(\kappa_2(G_{m \times m})   > m \, . \, x \right)
< \frac{2.60 }{x},
\end{eqnarray*}
which is consistent with (4.4) except that we improved the upper 
bound for the constant $C^\prime$ from 5.60 to 2.60.
From the following (4.5), we know that the constant $C^\prime$
in (4.4) actually must at least be 2. 

\item For the special case of large real random $m \times m$ matrices,
it has been proved in \cite{Edelman89} that
\begin{eqnarray*}
\lim_{m \rightarrow \infty}   P\left( \frac{\kappa_2(G_{m \times m})} {m}  < x \right)
= e^{-\frac{2}{x} - \frac{2}{x^2} }.
\end{eqnarray*}
Therefore, we have
\begin{eqnarray}
\lim_{m \rightarrow \infty}   P\left( \frac{\kappa_2(G_{m \times m})} {m}  > x \right)
= 1- e^{-\frac{2}{x} - \frac{2}{x^2} }  \sim   \frac{2}{x}
\end{eqnarray}
as $x \rightarrow \infty$.
Hence, the smallest possible universal constant $C$ in Theorem 4.5 
must be no smaller than $2 \sqrt{2 \pi}$. Therefore,
the universal constant $C$ in Theorem 4.5  actually must satisfy
\begin{eqnarray}
C \geq 2 \sqrt{2 \pi} \approx 5.013.
\end{eqnarray}
\end{remunerate}

Similar to real random matrices, for complex random matrices, we have
the following Theorem 4.6. Theorem 4.6 can be proved using the 
same techniques as Theorem 4.5, so we 
will omit the proof and only give the result.

\begin{theorem} \label{ Complex Case: new upper bounds for tail }
For any $n \geq m \geq 2$ and $x \geq n-m+1$, the 2-norm condition
number of $\widetilde{G}_{m \times n}$
satisfies
\begin{eqnarray*}
P\left(\frac{\kappa_2(\widetilde{G}_{m \times n}) } {{n}/{(n-m+1)}}  > x\right) < 
\frac{1}{{2\pi}} \left( \frac{\widetilde{C} }{x} \right)^{2(n-m+1)},
\end{eqnarray*}
where $\widetilde{C}  \leq 6.298$ is a universal positive constant independent 
of $x, m$, and $n$.
\end{theorem}

\section{The lower bounds for the distribution tails}

In this section, we will prove the lower bounds for the tails of 
the condition number distributions of random rectangular matrices
whose elements are independent and identically distributed 
standard normal random variables. Our main results are Theorem 5.5
for real random matrices, and Theorem 5.6 for complex random matrices.

\begin{lemma} \label{Real case: prob inequality C}
For any $ B > 0$, $x > 0 $, and $n \geq m \geq 2 $, 
the smallest eigenvalue $\lambda_{min}$ of $W_{m,n}$ satisfies
\begin{eqnarray*}
P \left(  \lambda_{min} \leq \frac{B^2 n}{x^2} \right) 
>  \sqrt{ \frac{2e^{\frac{5}{6}} } {3} } e^{-\frac{B^2 mn}{2 x^2}}
   \frac{1}{ \Gamma(n-m+2) } \left( \frac{e^{-\frac{1}{2}}B n}{x} \right)^{n-m+1}.  
\end{eqnarray*}
\end{lemma}

\begin{proof}
From the lower bound for the probability density function of
$\lambda_{min}$ in Lemma 3.3, we have
\begin{eqnarray*}
P \left(\lambda_m \leq \frac{B^2 n}{x^2} \right) 
&   =   & \int_0^{\frac{B^2 n}{x^2}}  f(\lambda_m)  d \lambda_m  \\
&   >   & \int_0^{\frac{B^2 n}{x^2}}  L_{m,n} 
                  e^{-\frac{m}{2}\lambda_m } \lambda_m^{\frac{1}{2}(n-m-1)}  d \lambda_m  \\
&   >   &   L_{m,n} e^{-\frac{B^2 m n}{2x^2}}
            \int_0^{\frac{B^2 n}{x^2}}  \lambda_m^{\frac{1}{2}(n-m-1)}  d \lambda_m  \\
&   =   & L_{m,n} e^{-\frac{B^2 m n}{2x^2}}
          \frac{2 n^{\frac{n-m+1}{2}}}{n-m+1} \left( \frac{B}{x} \right)^{n-m+1} \\
&   =   &  e^{-\frac{B^2 m n}{2x^2}}  \frac{\Gamma\left(\frac{n+1}{2}\right)} 
           {\Gamma\left(\frac{m}{2}\right)\left(\frac{n}{2}\right)^{\frac{n-m+1}{2}}} 
                  \frac{1}{ \Gamma(n-m+2) } \left( \frac{Bn}{x} \right)^{n-m+1}. 
\end{eqnarray*}
Note that
\begin{eqnarray*} 
\frac{n+1}{2} \Gamma\left(\frac{n+1}{2}\right)
> \sqrt{2 \pi} \left(\frac{n+1}{2}\right)^{\frac{n+2}{2}} e^{-\frac{n+1}{2}}, 
\end{eqnarray*}
and 
\begin{eqnarray*} 
\frac{m}{2} \Gamma\left(\frac{m}{2}\right)
< \sqrt{2 \pi} \left(\frac{m}{2}\right)^{\frac{m+1}{2}} e^{-\frac{m}{2} + \frac{1}{6m}}. 
\end{eqnarray*}
Therefore
\begin{eqnarray*} 
\frac{\Gamma\left(\frac{n+1}{2}\right)} 
{\Gamma\left(\frac{m}{2}\right)\left(\frac{n}{2}\right)^{\frac{n-m+1}{2}}} 
&   >   &  e^{-\frac{n-m+1}{2} -\frac{1}{6m}} \sqrt{ \frac{(n+1)^n}{m^{m-1} n^{n-m+1}} } \\
&   =   &  e^{-\frac{n-m+1}{2} -\frac{1}{6m}} \sqrt{ \frac{n^{n+1}(1+1/n)^{n+1}}{(n+1) m^{m-1} n^{n-m+1}} } \\
&   >   &  e^{-\frac{n-m+1}{2} -\frac{1}{6m}} \sqrt{ \frac{n e }{n+1} }.
\end{eqnarray*}
Since $2 \leq m \leq n$, therefore, we have
\begin{eqnarray*} 
\frac{\Gamma\left(\frac{n+1}{2}\right)}
{\Gamma\left(\frac{m}{2}\right)\left(\frac{n}{2}\right)^{\frac{n-m+1}{2}}} 
&   >   &  \sqrt{\frac{2 e^{ \frac{5}{6} } }{3}} e^{- \frac{n-m+1}{2}}.
\end{eqnarray*}
Therefore, we have 
\begin{eqnarray*}
P \left(\lambda_m \leq \frac{B^2 n}{x^2} \right) 
& >  & \sqrt{\frac{2 e^{ \frac{5}{6} } }{3}} e^{-\frac{B^2 m n}{2 x^2} } 
       \frac{1}{\Gamma(n-m+2) } \left( \frac{e^{-\frac{1}{2}}B n}{x} \right)^{n-m+1}.  
\end{eqnarray*}
\qquad
\end{proof}

Similar to real random matrices, we have
the following Lemma 5.2 for complex random matrices. Lemma 5.2 can be proved using the 
same techniques as Lemma 5.1, so we 
will omit the proof and only give the result.

\begin{lemma} \label{Complex case: prob inequality C}
For any $B > 0$, $x > 0 $, and $2 \leq m \leq n $, the smallest eigenvalue $\widetilde{\lambda}_{min}$ 
of $\widetilde{W}_{m,n}$ satisfies
\begin{eqnarray*}
P \left(\widetilde{\lambda}_{min} \leq \frac{B^2 n}{x^2} \right) 
>  e^{1-\frac{1}{12m}} e^{-\frac{B^2 mn}{2 x^2}}
   \frac{1}{ \Gamma(n-m+2)^2 } \left( \frac{e^{-1}B^2 n^2}{2x^2} \right)^{n-m+1}.  
\end{eqnarray*}
\end{lemma}

The proof of the following Lemma 5.3 is based on 
the upper bound of the joint probability density 
function of $\lambda_{max}$ and $\lambda_{min}$ 
in Lemma 3.1 and the upper bound of the incomplete
Gamma function in Lemma 2.5.

\begin{lemma} \label{Real case: prob inequality D}
For any $B\leq e^{-1.7}$, $x>0$, and $2 \leq m \leq n$, the largest eigenvalue
$\lambda_{max}$  and the smallest eigenvalue $\lambda_{min}$ 
of $W_{m,n}$ satisfy
\begin{eqnarray*}
P \left(\lambda_{min} \leq \frac{B^2 n}{x^2},\frac{\lambda_{max}}{\lambda_{min}} \leq x^2 \right)
&  <  & \frac {11 B^{m-1}} {4 \sqrt{4\pi}}  
        \frac{1}{\Gamma(n-m+2)} \left( \frac{ e^{-\frac{1}{2}}Bn}{x} \right)^{n-m+1}. 
\end{eqnarray*}
\end{lemma}

\begin{proof}
From the upper bound for the joint probability density function of
$\lambda_{max}$ and $\lambda_{min}$ in Lemma 3.1, we have 
\begin{eqnarray*}
 P \left(\lambda_{min} \leq \frac{B^2 n}{x^2},\frac{\lambda_{max}}{\lambda_{min}} \leq x^2 \right)
&   =   & \int_0^{\frac{B^2 n}{x^2}} \int_0^{t x^2}  f_{\lambda_{max}, \lambda_{min} }(s,t) d s d t  \\
&   <   & C_{m,n} \int_0^{\frac{B^2 n}{x^2}} \int_0^{t x^2}  
            e^{-\frac{1}{2} t } t^{\frac{1}{2}(n-m-1)} 
            e^{-\frac{1}{2} s } s^{\frac{1}{2}(n+m-3)}    d s d t. 
\end{eqnarray*}
Taking the transform $u=t x^2$, we have
\begin{eqnarray*}
 P \left(\lambda_{min} \leq \frac{B^2 n}{x^2},\frac{\lambda_{max}}{\lambda_{min}} \leq x^2 \right)
&   =   & C_{m,n}  \left( \frac{1}{x} \right)^{n-m+1} 
          \int_0^{B^2 n}  e^{-\frac{u}{2 x^2}} u^{\frac{1}{2}(n-m-1)} \\ & & \qquad
          \left( \int_0^u e^{-\frac{1}{2}s } s^{\frac{1}{2}(n+m-3)} ds \right) d u. 
\end{eqnarray*}
According to Lemma 2.5, if $u \leq n+m-3$, then
\begin{eqnarray*}
\int_0^u e^{-\frac{1}{2}s } s^{\frac{1}{2}(n+m-3)} ds
&  \leq  &  e^{-\frac{1}{2}u } u^{\frac{1}{2}(n+m-1)}.
\end{eqnarray*}
Therefore, when $B \leq e^{-1.7}$, we have
\begin{eqnarray*}
 P \left(\lambda_{min} \leq \frac{B^2 n}{x^2},\frac{\lambda_{max}}{\lambda_{min}} \leq x^2 \right)
&  \leq  & C_{m,n}  \left( \frac{1}{x} \right)^{n-m+1} 
           \int_0^{B^2 n} e^{-\frac{u}{2 x^2} -\frac{1}{2}u } u^{n-1} d u \\
&  \leq  & C_{m,n}  \left( \frac{1}{x} \right)^{n-m+1}  \int_0^{B^2 n}  e^{-\frac{1}{2}u } u^{n-1} d u.
\end{eqnarray*}
Since $B \leq e^{-1.7}$, so $B^2 n \leq 2(n-1)$. Applying Lemma 2.5 again, we have
\begin{eqnarray*}
P \left(\lambda_{min} \leq \frac{B^2 n}{x^2},\frac{\lambda_{max}}{\lambda_{min}} \leq x^2 \right)
&  \leq  & C_{m,n}  \left( \frac{1}{x} \right)^{n-m+1}  e^{-\frac{B^2 n}{2} } B^{2n} n^n \\
&   =   &\frac{ e^{-\frac{B^2 n}{2} } B^{n+m-1} n^{m-1}  }{4\Gamma(m-1)\Gamma(n-m+1) }
                       \left( \frac{Bn}{x} \right)^{n-m+1}.   
\end{eqnarray*}
From (4.2), we have
\begin{eqnarray*}
\frac{n^{m-2}}{\Gamma(m-1) }
&   <   & \frac {e^{n}}{\sqrt{4 \pi}}. 
\end{eqnarray*}
Therefore, we have
\begin{eqnarray*}
P \left(\lambda_{min} \leq \frac{B^2 n}{x^2},\frac{\lambda_{max}}{\lambda_{min}} \leq x^2 \right)
&  \leq  &  \frac{ e^n e^{-\frac{B^2 n}{2} } B^{n+m-1} n  }{4 \sqrt{4 \pi} \Gamma(n-m+1) }
                       \left( \frac{Bn}{x} \right)^{n-m+1}   \\
&  \leq  & \frac{B^{m-1} n (n-m+1) e^{\frac{3}{2}n} e^{-\frac{B^2 n}{2} } B^{n} } {4 \sqrt{4\pi} }  \\ & & \qquad
                      \frac{1}{\Gamma(n-m+2)} \left( \frac{ e^{-\frac{1}{2}}Bn}{x} \right)^{n-m+1}.
\end{eqnarray*}
When $B \leq e^{-1.7}$, for all $n \geq m \geq 2$, we have
\begin{eqnarray*}
n (n-m+1)  e^{\frac{3}{2}n} e^{-\frac{B^2 n}{2} } B^{n} 
&  <  & 11.
\end{eqnarray*}
Therefore,when $B \leq e^{-1.7}$, we have
\begin{eqnarray*}
P \left(\lambda_{min} \leq \frac{B n}{x^2},\frac{\lambda_{max}}{\lambda_{min}} \leq x^2 \right)
&  <  & \frac {11 B^{m-1}} {4 \sqrt{4\pi}}  
        \frac{1}{\Gamma(n-m+2)} \left( \frac{ e^{-\frac{1}{2}}Bn}{x} \right)^{n-m+1}. 
\end{eqnarray*}
\qquad
\end{proof}

Similar to real random matrices, we have
the following Lemma 5.4 for complex random matrices. 
Lemma 5.4 can be proved using the 
same techniques as Lemma 5.3, so we 
will omit the proof and only give the result.

\begin{lemma} \label{Complex case: prob inequality D}
For any $B^2 \leq e^{-1.2}$, $x>0$, and $2 \leq m \leq n$, the largest eigenvalue
$\widetilde{\lambda}_{max}$  and the smallest eigenvalue $\widetilde{\lambda}_{min}$ 
of  $\widetilde{W}_{m,n}$ satisfy
\begin{eqnarray*}
P \left(\widetilde{\lambda}_{min} \leq \frac{B n}{x^2},
        \frac{\widetilde{\lambda}_{max}}{\widetilde{\lambda}_{min}}  \leq x^2 \right)
&  <  & 0.0352 \frac{1}{\Gamma(n-m+2)^2} \left( \frac{ e^{-1}B^2n^2}{2x^2} \right)^{n-m+1}. 
\end{eqnarray*}
\end{lemma}

We are now prepared to derive the lower bounds for
the tails of the condition number distributions of
random matrices whose elements are 
independent and identically distributed standard 
normal random variables

\begin{theorem} \label{ Real Case: new lower bounds for tail }
For any $x \geq n-m+1$ and $n \geq m \geq 2$, the 2-norm condition number of $G_{m \times n}$
satisfies
\begin{eqnarray}
P\left(\frac{\kappa_2(G_{m \times n}) } {{n}/{(n-m+1)}}  > x\right)  
> \frac{1} {\sqrt{2 \pi}} \left(\frac{c}{x} \right)^{n-m+1}, 
\end{eqnarray}
where $c\geq 0.245 $ is a universal positive constant independent 
of $x, m$, and $n$.
\end{theorem}

\begin{proof}
For any positive constant $H$, we have
\begin{eqnarray*}
P( \kappa_2(G_{m \times n}) > x ) 
&  =  &  P \left( \frac{\lambda_1}{\lambda_m} > x^2 \right)   \\
&  >  &  P \left( {\lambda_m} \leq \frac{H^2n}{x^2}, \frac{\lambda_1}{\lambda_m} > x^2 \right)   \\
&  =  &  P \left( {\lambda_m} \leq \frac{H^2n}{x^2} \right)  
        -P \left( {\lambda_m} \leq \frac{H^2n}{x^2}, 
                  \frac{\lambda_1}{\lambda_m} \leq x^2 \right). 
\end{eqnarray*}
Let $H = e^{-1.7}$, then based on Lemma 5.1 and Lemma 5.3, we have
\begin{eqnarray*}
P( \kappa > x ) 
&  >  & \left( \sqrt{ \frac{2e^{\frac{5}{6}} } {3} } e^{-\frac{H^2 m n}{2 x^2}} 
          - \frac {11 H^{m-1}} {4 \sqrt{4\pi}} \right) 
         \frac{1}{ \Gamma(n-m+2) } \left( \frac{ e^{-\frac{1}{2}}H n}{x} \right)^{n-m+1}.
\end{eqnarray*}
From Lemma 2.7, we have
\begin{eqnarray*}
 \Gamma(n-m+2)  < \sqrt{2 \pi (n-m+1)} (n-m+1)^{n-m+1} e^{-(n-m+1)+\frac{1}{12(n-m+1)}}.
\end{eqnarray*}
Note that, for $2 \leq m \leq n$, we have
\begin{eqnarray*}
\sqrt{n-m+1} < 1.21^{n-m+1} \mbox{, and } \frac{1}{12(n-m+1)} \leq \frac{1}{12},
\end{eqnarray*}
Therefore, we have
\begin{eqnarray*}
P\left(\kappa_2(G_{m,n})  > x\right)  
&  >  &  \left( \sqrt{ \frac{2e^{\frac{5}{6}} } {3} } e^{-\frac{H^2 m n}{2 x^2}}
            - \frac {11 H^{m-1}} {4 \sqrt{4\pi}} \right)
          \frac{e^{- \frac{1}{12}}} {\sqrt{2 \pi}}  
         \left( \frac{\frac{e}{1.21 (n-m+1)} e^{-\frac{1}{2}}Hn } {x} \right)^{n-m+1}.
\end{eqnarray*}
Since $H=e^{-1.7}$, $x \geq 1$, and $2 \leq m \leq n$, so we have
\begin{eqnarray*}
 \left( \sqrt{ \frac{2e^{\frac{5}{6}} } {3} }   e^{-\frac{H^2 m n}{2 x^2}}
   - \frac {11 H^{m-1}} {4 \sqrt{4\pi}} \right) 
  e^{- \frac{1}{12}}
&  >  & 0.99.
\end{eqnarray*}
Therefore, we have
\begin{eqnarray*}
P\left(\kappa_2(G_{m,n})  > x\right)  
&  >  &   \frac{0.99} {\sqrt{2 \pi}} \left( \frac{0.248 \frac{n}{n-m+1} }{x} \right)^{n-m+1} \\
&  >  &   \frac{1} {\sqrt{2 \pi}} \left( \frac{0.245 \frac{n}{n-m+1} }{x} \right)^{n-m+1}.
\end{eqnarray*}
Therefore
\begin{eqnarray*}
P\left(\frac{\kappa_2(G_{m,n}) } {{n}/{(n-m+1)}}  > x\right)  
&  >  &  \frac{1} {\sqrt{2 \pi}}\left( \frac{0.245 }{x} \right)^{n-m+1}. 
\end{eqnarray*}
Let $c = 0.245$, then we get (5.1).
\qquad
\end{proof}

Remark:
\begin{remunerate}

\item 
The lower bound in Theorem 5.5 is for arbitrary $n \geq m \geq 2$ 
and $x \geq n-m+1$. 
For some special case of $m$ and $n$, more precise lower bound can be obtained. 
For example, for the special case of real random $m \times m$ matrices,
where $m \geq 3$, 
it has been proved in \cite{Azais03} that
\begin{eqnarray*}
P \left(\kappa_2(G_{m \times m})   >  m \, .  \, x   \right)
> \frac{c }{x},
\end{eqnarray*}
where $c \geq 0.13$ is a universal positive constant independent 
of $x$ and $m$.
 
In Theorem 5.5, however, if we take $m = n$, then we can only get
\begin{eqnarray*}
P\left(\kappa_2(G_{m \times m})   > m \, . \, x \right)
> \frac{0.097 }{x},
\end{eqnarray*}

\item For the special case of real random $2 \times n$ matrices, 
based on the exact probability density function of $\kappa_2(G_{2 \times n})$ 
in \cite{Edelman89}, we can get
\begin{eqnarray*}
P\left(\kappa_2(G_{2 \times n}) > x \right) = \left( \frac{2x}{x^2+1} \right)^{n-1}
\sim  \left( \frac{2}{x} \right)^{n-1}
\end{eqnarray*}
as $x \rightarrow \infty$. Hence, the constant $c$ in Theorem 5.5 
is no larger than $2$.
Therefore, the constant $c$ in Theorem 5.5 actually satisfies 
\begin{eqnarray}
0.245 \leq c \leq 2.
\end{eqnarray}
\end{remunerate}

Similar to real random matrices, we have
the following Theorem 5.6 for complex random matrices. 
Theorem 5.6 can be proved using the 
same techniques as Theorem 5.5, so we 
will omit the proof and only give the result.

\begin{theorem} \label{ Complex Case: new lower bounds for tail }
For any  $x \geq n-m+1$ and $n \geq m \geq 2$, the 2-norm condition number of $G_{m \times n}$
satisfies
\begin{eqnarray*}
P\left(\frac{\kappa_2(\widetilde{G}_{m \times n}) } {{n}/{(n-m+1)}}  > x\right)  
> \frac{1} {2 \pi} \left(\frac{c}{x} \right)^{2(n-m+1)}, 
\end{eqnarray*}
where $c\geq 0.319 $ is a universal positive constant independent 
of $x, m$, and $n$.
\end{theorem}

\section{The upper bounds for the expected logarithms}

For square Gaussian random matrix $G_{n \times n}$, in \cite{Smale85}, Smale asked
for $E ( \log \kappa_2(G_{n \times n}))$. Similarly, for rectangular Gaussian random 
matrix $G_{m \times n}$, it is also interesting to investigate $E ( \log \kappa_2(G_{m \times n}))$.
In this section, we will derive upper bounds for $E ( \log \kappa_2(G_{m \times n}))$
and $E ( \log \widetilde{\kappa}_2(G_{m \times n}))$.
Our main results are Theorem 6.1
and Theorem 6.2. 

\begin{theorem} \label{ bounds for mean }
For any $n \geq m \geq 2$, the 2-norm condition number of $G_{m \times n}$
satisfies
\begin{eqnarray}
E(\log\kappa_2(G_{m \times n})) < \log\frac{n}{n-m+1} + 2.258.
\end{eqnarray}
\end{theorem}

\begin{proof}
Let  $f_{\kappa}(x)$ be the probability density function of $\kappa_2(G_{m \times n})$, then
\begin{eqnarray*}
E \log \left( \frac{\kappa_2(G_{m \times n})} {6.414 \frac{n}{n-m+1} } \right)
&  =  &  \int_1^\infty \log \left( \frac{x} {6.414 \frac{ n}{n-m+1} } \right)  f_{\kappa}(x) dx  \\
&  <  &  \int_{6.414 \frac{n}{n-m+1}}^\infty \log \left( \frac{x} {6.414\frac{ n}{n-m+1} } \right) 
          f_{\kappa}(x) dx      \\
&  =  &  \int_{6.414 \frac{ n}{n-m+1}}^\infty  P(\kappa_2(G_{m \times n})>x) \quad \frac{1}{x} \quad dx.  
\end{eqnarray*}

From Theorem 4.3, we have
\begin{eqnarray*}
P\left(\kappa_2(G_{m \times n}) > x\right) 
< \frac{1}{\sqrt{2\pi}} \left( \frac{6.414 \frac{n}{n-m+1}}{x} \right)^{n-m+1}.
\end{eqnarray*}

Therefore, we have
\begin{eqnarray*}
E \log\left( \frac{\kappa_2(G_{m \times n})} {6.414 \frac{n}{n-m+1} } \right)
&  <  &  \frac{1}{\sqrt{2\pi}}\int_{6.414 \frac{ n}{n-m+1}}^\infty  
         \left( \frac{6.414 \frac{ n}{n-m+1}}{x} \right)^{n-m+1} \quad \frac{1}{x} \quad dx   \\
&  =  &   \frac{1}{ (n-m+1) \sqrt{2\pi} }    \\
&  <  &   0.399.
\end{eqnarray*}

Therefore, we have
\begin{eqnarray*}
E \log ( \kappa_2(G_{m \times n}) ) 
&  <  & \log  \frac{n}{n-m+1}  + \log 6.414 + 0.399  \\
&  <  &  \log  \frac{n}{n-m+1} + 2.258.  
\end{eqnarray*}
\qquad
\end{proof}

Remark:
\begin{remunerate}

\item For the special case of real random $m \times m$ matrices,
from the results in \cite{Szarek91}, we can get
\begin{eqnarray}
E \log ( \kappa_2(G_{m \times m}) ) \leq \log m + \frac{3+3\log2}{2} \approx 2.54.
\end{eqnarray}

In Theorem 6.1, if we take $m = n$, then we have
\begin{eqnarray*}
E \log ( \kappa_2(G_{m \times n}) ) <  \log n + 2.258.
\end{eqnarray*}
which is a slightly improved version of (6.2).

\item The upper bound in Theorem 6.1 is for arbitrary $n \geq m \geq 2$. 
For some special case of $m$ and $n$ or large $m$ and $n$, more precise results exist: 

For the special case of real random $2 \times n$ matrices, it was
shown in \cite{Edelman88} that
\begin{eqnarray*}
E \log ( \kappa_2(G_{2 \times n}) ) = \frac{1}{2} \sqrt{\pi} 
\frac{\Gamma \left(\frac{n-1}{2} \right)}  {\Gamma \left(\frac{n}{2} \right)}.
\end{eqnarray*}

For real random $m \times m$ matrices,
it has been proved in \cite{Edelman88} that
\begin{eqnarray*}
E \log ( \kappa_2(G_{m \times m}) ) = \log m + c + o(1)
\end{eqnarray*}
as $m \rightarrow \infty$, where $c \approx 1.537$.

For rectangular matrix $G_{m_n \times n}$, if $\lim_{n \rightarrow \infty} m_n/n = y$
and $0 < y < 1$, then it has been proved in \cite{Edelman88} that
\begin{eqnarray*}
E \log ( \kappa_2(G_{m_n \times n}) ) = \log \frac{1+\sqrt{y}}{1-\sqrt{y}} + o(1)
\end{eqnarray*}
as $n \rightarrow \infty$

\end{remunerate}

Similar to real random matrices, we have
the following Theorem 6.2 for complex random matrices. 
Theorem 6.2 can be proved using the 
same techniques as Theorem 6.1, so we 
will omit the proof and only give the result.

\begin{theorem} \label{ Complex case: bounds for mean }
For any $n \geq m \geq 2$, the 2-norm condition number of $G_{m \times n}$
satisfies
\begin{eqnarray*}
E(\log\kappa_2(\widetilde{G}_{m \times n})) < \log\frac{n}{n-m+1} + 2.240.
\end{eqnarray*}
\end{theorem}

\section*{Acknowledgments}
We would like to thank Julien Langou for precious discussions.
We are very grateful to referees for their helpful corrections,
valuable comments, and constructive suggestions which have improved
the presentation of this paper.


\begin{thebibliography}{18} 

\bibitem{Abram70} {\sc M. Abramowitz and I. A. Stegun, eds.}, 
{\em Handbook of Mathematical Functions}, Dover, New York, 1970. 

\bibitem{Azais03} {\sc J. M. Azais and M.Wschebor}, {\em Upper and lower
bounds for the tails of the distribution of the condition number
of a gaussian matrix}, submitted.

\bibitem{Chen104} {\sc Z. Chen and J. Dongarra}, {\em Numerically Stable Real Number
Codes Based on Random Matrices}, Technical Report ut-cs-04-526, Department
of Computer Science, The University of Tennessee, Knoxville, 2004

\bibitem{Chen204} {\sc Z. Chen and J. Dongarra}, {\em Fault tolerant high performance computing
by a coding approach}, in preparation.

\bibitem{Edelman89} {\sc A. Edelman}, {\em Eigenvalues and condition numbers of random matrices},
Ph.D. thesis, Dept. of Math., M.I.T., 1989.

\bibitem{Edelman88} {\sc A. Edelman}, {\em Eigenvalues and condition numbers of random matrices},
SIAM J. of Matrix Anal. and Applic., 9 (1988), pp. 543-560.

\bibitem{Edelman05} {\sc A. Edelman and B. Sutton}, {\em Tails of condition number distributions},
Submitted to SIAM J. of Matrix Anal. and Applic..

\bibitem{Graham94} {\sc R. L. Graham, D. E. Knuth,  and O. Patashnik}, {\em Concrete Mathematics: 
A Foundation for Computer Science}, 2nd ed.,  Reading, MA: Addison-Wesley, 1994.

\bibitem{James64} {\sc A. T. James}, {\em Distributions of matrix variates and latent roots
derived from normal samples}, Ann. Math. Statist., 35 (1964), 99. 475-501.

\bibitem{Neumann63} {\sc J. von Neumann and H. H. Goldstine}, {\em Numerical inverting 
of matrices of higher order}, in John von Neumann, Collected Works, Vol. 5: Design of Computers,
Theory of Automata and Numerical Analysis, A. H. Taub, ed., Pergamon, New York, 1963.

\bibitem{Smale85}  {\sc S. Smale}, {\em On the efficiency of Algorithms of analysis}, 
Bull. Amer. Math. Soc., 13 (1985), pp. 87-121.

\bibitem{Szarek91}  {\sc S. J., Szarek}, {\em Condition numbers of random matrices},
J. Complexity 7 (1991), no. 2, pp.131--149.

\bibitem{Weiss86}   {\sc N. Weiss, G. W. Wasilkowski, H. Wo\'zniakowski, and M. Shub},
{\em Average condition number for solving linear equations}, Linear Algebra Appl., 83 (1986), 
pp. 79-102.

\end{thebibliography}



\end{document}